\documentclass[12pt, reqno]{amsart}
\usepackage{mathrsfs}
\usepackage{amsfonts}
\usepackage{cite}
\usepackage{amsmath}
\usepackage{amsthm}
\usepackage{amssymb}
\usepackage{bm}
\usepackage{cases}
\numberwithin{equation}{section}
\newtheorem{thm}{Theorem}[section]

\newtheorem{lemma}{Lemma}[section]
\newtheorem{property}{Property}[section]
\newtheorem{corollary}{Corollary}[section]
\newtheorem{cor}{Corollary}[section]

\newtheorem{pro}{Proposition}[section]
\newtheorem{defi}{Definition}[section]

\usepackage[%
pdfstartview=FitH, %
colorlinks=true, %
linkcolor=blue, %
citecolor=red]{hyperref}

\begin{document}
\title[An affine isoperimetric  inequality for log-concave functions]{An affine isoperimetric  inequality for log-concave functions${}^*$}

\author[Zengle Zhang and Jiazu Zhou]  {Zengle Zhang$^{1}$ and Jiazu Zhou$^{2,3,**}$}

\address{\parbox[l]{1\textwidth}{
1 Key Laboratory of Group and Graph Theories and Applications, Chongqing University of Arts and
Sciences, Chongqing 402160, China.\\
\email{zhangzengle128@163.com} \\
2 School of Mathematics and Statistics, Southwest University, Chongqing 400715, China.\\
3 College of Science, Wuhan University of Science and Technology, Wuhan 430081, China. \\
\email{zhoujz@swu.edu.cn}
}}


\thanks{*Supported in part by NSFC (No. 12071378, No. 12301071) and the Science and Technology Research Program of Chongqing Municipal Education Commission (No. KJQN202201339); }
\thanks{**The corresponding author}

\maketitle

\begin{abstract}   The authors gave an affine isoperimetric inequality \cite{LYZ2010} that gives a lower bound for the volume of a polar body and the equality holds if and only if the body is a simplex. In this paper, we give  a functional isoperimetric   inequality for log-concave functions that contains the affine isoperimetric inequality  of Lutwak, Yang and Zhang   in \cite{LYZ2010}.
 \end{abstract}

\vskip 0.5cm

\section{Introductions}
The isoperimetric problem
was known in Ancient Greece.
However, the first mathematically rigorous proof was obtained only
in the 19th century by Weierstrass based on works of Bernoulli, Euler, Lagrange and others.
The isoperimetric problem is equivalent to the isoperimetric inequality that is bounding by the surface area and volume of the geometric domain $K$ in the Euclidian space $\mathbb {R}^n$ with equality if and only if $K$ is a standard ball.
The natural generalizations of the isoperimetric inequality are Alexandrov-Fenchel inequalities that are bounding by mixed volumes of convex bodies in integral and convex geometry analysis.
Blascheke-Santal\'o inequality, differential affine isoperimetric inequality, Busemann-Petty centroid inequality, Petty projection inequality and more affine isoperimetric inequalities are found with  equalities  for ellipsoids, polar bodies or simplices \cite{Ball1991, CampiGronchi2002, CampiGronchi2002-2, HF2009, LYZ2000, LYZ2000Duke, LYZ2004, LYZ2007, LZ1997, MeyerWerner2006, SW2012}.

During past decades, the reverse affine isoperimetric inequalities, with simplices,  cubes or  polar bodies as extremal,  received more attentions.
 Usually, reverse affine isoperimetric inequalities are  much harder to establish.
One of the known  open problem about the reverse affine isoperimetric inequalities is the  Mahler conjecture.

 Let $\mathcal{K}_o^n$ denote the set of compact convex sets that contain the origin in their interiors in  $\mathbb{R}^n$.
Let $K\in\mathcal{K}_o^n$, the polar body  $K^\circ$  of $K$, is defined by
\begin{align}
K^\circ=\{x\in\mathbb{R}^n:  x\cdot y \le 1,\quad \text{for all}\quad y\in K \},
\end{align}
and $ x\cdot y$ denotes the inner product of $x$ and $y$.

 If $K\in\mathcal{K}_o^n$ is symmetric, Mahler\cite{Mahler1939} conjectured that
\begin{align}\label{In.mahlerC}
|K||K^\circ|\ge \frac{4^n}{n!},
\end{align}
where $|K|$ denotes the volume of $K$, with equality if and only if either $K$ or $K^\circ$ is a parallelepiped.

Inequality (\ref{In.mahlerC}) was proved by Mahler  \cite{Mahler1939} for $n=2$, the  $3$-dimensional case was solved by Iriyeh and Shibata \cite{IS2020}, and the case when $n\ge 4$ is still open. One can see \cite{Ball1991,Barthe1998,BBF14,CampiGronchi2002,CampiGronchi2002-2,GiannopoulosMilman2002,
GiannopoulosPapad2000,HF2009,LYZ2000,LYZ2000Duke,LYZ2004,LYZ2007,LYZ2010,LZ1997,MeyerWerner2006,SW2012,ZhuHongYe,ZouXiong2014,ZouXiong2016} for more references on Mahler's conjecture, affine isoperimetric inequalities and reverse affine isoperimetric inequalities.

Let $K\in\mathcal{K}_o^n$. The radial function of $\Gamma_{-2}K$, usually called the LYZ ellipsoid, is defined by
\begin{align}\label{radial-LYZ-ellipsoid}
\rho_{\Gamma_{-2}K}(u)^{-2}=\frac{1}{|K|}\int_{\partial K}| u\cdot n_K(x)|^2h_{K}(n_K(x))^{-1}d\mathcal{H}^{n-1}(x), \quad u\in S^{n-1},
\end{align}
where $n_K(x)$ denotes the outer unit normal vector of  at $x \in \partial K$, $h_K$ is the support function of $K$ and $\mathcal{H}^{n-1}$ is the $(n-1)$-dimensional Hausdorff measure.
Another known reverse affine isoperimetric inequality is the following {\em LYZ  polar inequality} \cite{LYZ2010}.
\bigskip

{\bf LYZ polar inequality.}
Let $K\in\mathcal{K}_o^n$, then
\begin{eqnarray}\label{In.LYZpolar}
|K^{\circ}||\Gamma_{-2}K|\geq\frac{(n+1)^{\frac{n+1}{2}}\omega_n}{n!n^{\frac{n}{2}}}
\end{eqnarray}
with equality if and only if $K$ is a simplex whose centroid is at the origin.
\bigskip

Recently, Fang and Zhou \cite{fangzhou1} introduced the LYZ ellipsoid of log-concave functions by solving an extremum problem. A function $f:\mathbb{R}^n\to\mathbb{R}$ is called a log-concave function if $\log f$ is concave. Log-concave functions would be considered as  extension of convex bodies. Usually, the results of convex bodies can be covered by taking  log-concave functions, 
such as the Gaussian function or the characteristic function of convex bodies.
The researches on log-concave functions originated from Ball's Phd thesis \cite{Ball1986}, in which the Blaschke-Santal\'{o} inequality of even log-concave function was established.
In recent years, some considerable progress has been made in establishing the geometric inequalities of log-concave function.
Colesanti and Fragal\`a\cite{CF13} established the Minkowski inequality of log-concave function via the Pr\'{e}kopa-Leindler inequality. Fang, Xing and Ye\cite{FXY20+} further gave $L_p$ Minkowski inequalities and solved the existence of the even $L_p$ Minkowski problem of log-concave functions for $p>1$.
John and L\"{o}wner ellipsoids for log-concave functions and their related inequalities were given in \cite{AMJV18,LSW19jga}.
Some affine isoperimetric inequalities of the affine surface areas of log-concave functions were obtained in \cite{CFGLSW16,CY16}. Other inequalities, such as P\'olya-Szeg\"o inequality, Rogers-Shephard inequality and its reverse and some stability results for log-concave functional inequalites can be found in \cite{ABM,AMJV16,AMJV18,AFS20,AKM04,AKSW12,BBF14,BD2021,Col06,FM07,FM08,IN2022,KM05,LSW19,LSW19jga,Lin17}.

 In this paper, for a convex function $\varphi:\mathbb{R}^n\to\mathbb{R}\cup\{\infty\}$, we use $\varphi^\ast$ to denote the Legendre transform of $\varphi$, and $f^\circ=e^{-\varphi^\ast}$ to denote the dual of a log-concave function $f$. The integral of $f$ on $\mathbb{R}^n$ is denoted by $J(f)$ and the domain of $\varphi$ is denoted by ${\rm dom}(\varphi)=\{x:\varphi(x)<\infty\}$.
For any function $f:\mathbb{R}^n\to\mathbb{R}$ and $t\in\mathbb{R}$, the upper level set of $f$ is defined by
$$
K_t(f)=\{x\in\mathbb{R}^n: f(x)\ge t\}.
$$
Let ${\rm LC}_n$ denote the class of all upper semi-continuous log-concave functions $f=e^{-\varphi}$ with ${\rm dom}(\varphi)=\mathbb{R}^n$ and $0<J(f)<\infty$.

Let $f \in \ {\rm LC}_n$,  for the LYZ ellipsoid $\Gamma_{-2}f$ of
 the log-concave function $f$,  defined in Definition \ref{Def.LYZEforLC}, we have
 the following affine isoperimetric inequality, the functional LYZ polar inequality  of  (\ref{In.LYZpolar}).
\bigskip

{\bf Theorem.}
Let $f=e^{-\varphi}\in{\rm LC}_n$ and $\mu_f$ be the surface area measure of $f$. Suppose that $\varphi^\ast(x)>0$ on $\mathbb{R}^n\backslash\{o\}$ and $-\log t=\sup\{\varphi^\ast(x/\varphi^\ast(x)):x\in {\rm supp} (\mu_f)\}$, then
\begin{eqnarray}\label{our-inequality}
|K_{t}(f^\circ)|J(\Gamma_{-2}f)
\geq\frac{8^{\frac{n}{2}}(n+1)^{\tfrac{n+1}{2}}\Gamma(\frac{n}{2}+1)\omega_n}{n!n^n}.
\end{eqnarray}
Moreover, if $\varphi(x)>\varphi(o)$ for all $x\in\mathbb{R}^n\backslash\{o\}$, then equality holds if and only if $f(x)=e^{-{\|x\|_K}+1}$ and $K$ is a simplex whose centroid is at the origin.
\bigskip

Note that the inequality (\ref{our-inequality})  is affine invariant, that is, the inequality (\ref{our-inequality})
does not change under compositions of $f$  with affine transformations of $\mathbb{R}^n$. Inequality (\ref{our-inequality})
includes the LYZ polar inequality as  a special case. In particular,
when $f(x)=e^{-{\|x\|_K}+1}$ and $K$ is a convex body contains origin in its interior, the inequality (\ref{our-inequality}) becomes the LYZ polar
inequality (\ref{In.LYZpolar}). A detailed derivation can be found in Section \ref{Sec.5}. The supremum $-\log t$ may achieve infinity
and in this case the inequality (\ref{our-inequality}) is strict. However,  $- \log t$ is often bounded and some examples of the function $f$ will be given in Section \ref{Sec.5}.

This paper is organized as follows. In Section \ref{Sec.2}, we provide some basics of convex bodies and  log-concave functions. The Ball-Barth inequality for isotropic embeddings, which is a key tool to establish our main theorem, is introduced in Section \ref{Sec.3}.
In Section \ref{Sec.4}, the definition of the LYZ ellipsoid of the log-concave function $\Gamma_{-2}f$ is given, and its properties is listed. In Section \ref{Sec.5}, we give the proof of the main theorem and show that LYZ's polar inequality is a direct result of inequality (\ref{our-inequality}) ((Theorem \ref{mainthm}).

%

\section{Preliminaries and notations}\label{Sec.2}

In this section, we list some  preliminaries and notations about convex bodies and log-concave functions. Good
general references for the theory of convex bodies and log-concave functions are provided by the books of Gardner \cite{Gardner1995}, Schneider\cite{SchneiderR2014} and Rockfellar \cite{Roc70}.

\subsection{Basics regarding convex bodies}

Let $e_1,\cdots,e_n$ denote the standard Euclidean basis in the $n-$dimensional Euclidean space $\mathbb{R}^n$. For $x,y\in\mathbb{R}^n$,
$x\cdot y$ stands for the inner product of $x,y$, and $|x|$ for the Euclidean norm
of $x$. Let $K$ be a convex body (non-empty compact convex set of $\mathbb{R}^n$) contains the origin in its interior. The Minkowski function of $K$ is defined by
\begin{align}
\|x\|_K=\inf\{\lambda\ge0: x\in \lambda K\}.
\end{align}
Clearly, if $K$ is the unit ball $B^n$ in $\mathbb{R}^n$, the Minkowski function $\|\cdot\|_K$ becomes the Euclidean norm.

For convex body $K$,
its support function $h_K:\mathbb{R}^n\to\mathbb{R}$ is defined by
\begin{align*}
h_K(x)=\max\{x\cdot y: y\in K\},
\end{align*}
and its radial function $\rho_K: \mathbb{R}^n\backslash\{o\}\to\mathbb{R}$ is defined by
\begin{align*}
\rho_K(x)=\max\{\lambda\ge0: \lambda x\in K\}.
\end{align*}
The support function and Minkowski function are homogeneous of degree $1$ while the radial
function is homogeneous of degree $-1$. From the definitions of the support, Minkowski and radial functions and the definition of the polar body, one can see that
\begin{align}\label{Eq.Minkowski-support-radial}
\|x\|_{K}=h_{K^\circ}(x)=\frac{1}{\rho_K(x)},
\end{align}
for any $x\neq o$.
Let $GL(n)$ denote the general linear group of $\mathbb{R}^n$.
For $T\in GL(n)$,
\begin{align}\label{Eq.(TK)o}
(TK)^\circ=T^{-t}K^\circ,
\end{align}
where $T^{-t}$ is the inverse of transpose of $T$.

\subsection{Functional setting}

Let $\varphi: \mathbb{R}^n\rightarrow \mathbb{R}\cup \{+\infty\}$.   If for every $x,y\in\mathbb{R}^n$ and $\lambda\in[0,1]$,
$$
\varphi((1-\lambda)x+\lambda y)\leq (1-\lambda)\varphi(x)+\lambda \varphi(y),
$$
then $\varphi$ is a convex function. Let
$
\text{dom}(\varphi)=\{x\in \mathbb{R}^n: \varphi(x)<\infty\}.
$
Clearly, $\text{dom}(\varphi)$ is a convex set for any convex function $\varphi$.
The \emph{Legendre trasformation}
of $\varphi$ is the convex function defined by
\begin{eqnarray}\label{Fenchel conjugate}
\varphi^*(y)=\sup_{x\in\mathbb{R}^n}\left\{ x\cdot y-\varphi(x)\right\}\quad\quad\forall y\in\mathbb{R}^n.
\end{eqnarray}
Clearly, $\varphi(x)+\varphi^*(y)\geq  x\cdot y$ for all $x, y\in\mathbb{R}^n $, there is an equality if and only if  $x\in\text{dom}(\varphi)$ and $y$ is the gradient of $\varphi$ at $x$. Hence,
 $$
 \varphi^*(\nabla \varphi(x))+\varphi(x)=  x\cdot \nabla \varphi(x).
 $$
The convex function $\varphi:\mathbb{R}^n\rightarrow \mathbb{R}\cup \{+\infty\}$ is \emph{lower semi-continuous}, if the subset $\{x\in \mathbb{R}^n: \varphi(x)\le t\}$ is a closed set for any $t\in(-\infty,+\infty]$. If $\varphi$ is a lower semi-continuous  convex function, then also $\varphi^*$ is a lower semi-continuous  convex function, and $\varphi^{**}=\varphi$.



%
%


Let $f$ be a log-concave function.
The total mass of $f$ is defined by
\begin{align*}
J(f)=\int_{\mathbb{R}^n}f(x)dx.
\end{align*}
For $f=e^{-\varphi},g=e^{-\psi}$ and  $\alpha,\beta>0$,
the Asplund sum of $f$ and $g$ is defined by
\begin{eqnarray*}\label{addtion on log-concave function}
\alpha\cdot f\oplus\beta\cdot g=e^{-(\alpha\varphi^\ast+\beta\psi^\ast)^\ast}.
\end{eqnarray*}
Let $f=e^{-\varphi},g=e^{-\psi}\in{\rm LC}_n$. Rotem \cite{Rot22} gives the following first variational formula
\begin{eqnarray}\label{definition of the first variation}
\delta J(f,g)=\lim_{t\rightarrow 0^{+}}\frac{J(f\oplus t\cdot g)-J(f)}{t}
=\int_{\mathbb{R}^n}\psi^\ast(x)d\mu_f(x),
\end{eqnarray}
where $\mu_f$ is the surface area measure of $f$, which is defined by
$$
\mu_f(A)=\int_{\nabla\varphi(x)\in A}f(x)dx
$$
for any $A\subset\mathbb{R}^n$.
%

\section{The Ball-Barth inequality}\label{Sec.3}

Let $\mu$ be a positive Borel measure $\mu$ on $S^n$. A positive semi-definite $n\times n$ matrix $[\mu]$ can be generated by $\mu$, which is defined by
\begin{eqnarray}\label{def.matrix}
[\mu]=\int_{S^n}u\otimes ud\mu(u)
\end{eqnarray}
where $u\otimes u$ is the rank $1$ matrix generated by $u\in S^n$, or, equivalently by
\begin{eqnarray}\label{positive semi-definite}
v\cdot[\mu]v=\int_{S^n}|v\cdot u|^2d\mu(u)
\end{eqnarray}
for all $v\in S^n$. A positive Borel measure $\mu$ is said to be an isotropic measure if
\begin{align}\label{isotropic=m}
[\mu]=I_n,
\end{align}
where $I_n$ is the identity matrix. If $\mu$ is isotropic, summing the equation (\ref{positive semi-definite}) with $v=e_i, i=1,\cdots,n+1$, one has
\begin{eqnarray}\label{trace of positive semi-definite}
\mu(S^{n})=n+1.
\end{eqnarray}

The following Ball-Barth inequality was established in \cite{LYZ2004}.
\vskip 0.2 cm
\textbf{The Ball-Barth inequality.} \emph{Let $\mu$ an isotropic measure on $S^n$ and $l: S^n\rightarrow (0,\infty)$ be a continuous function. Then
\begin{eqnarray}\label{Ball-Barth-I}
\det\int_{S^n}l(u)u\otimes ud\mu(u)\geq \exp\left\{\int_{S^n}\log l(u)d\mu(u)\right\},
\end{eqnarray}
with equality if and only if $l(u_1) \cdots l(u_{n+1})$ is a constant for linearly independent $u_1,\cdots, u_{n+1}$ in ${\rm supp}(\nu)$.}
\vskip 0.2 cm

The isotropic embedding from $\mathbb{R}^n$ to $S^n$ is given as follows.

\begin{defi}\label{Def.iso-embedding}
Let $(\mathbb{R}^n,\mu)$ be a Borel measure space.   A continuous function  $h:\mathbb{R}^n\rightarrow S^{n}$   is  an isotropic embedding of $(\mathbb{R}^n,\mu)$ into $S^n$ if
\begin{eqnarray}\label{def.eq.iso-em}
\int_{\mathbb{R}^n}|u\cdot h(x)|^2d\mu(x)=1
\end{eqnarray}
for all $u\in S^{n}$.
\end{defi}

Let $\mu$ and $h$ have been given in Definition \ref{Def.iso-embedding}.
Summing the equation (\ref{def.eq.iso-em}) with $u=e_1,\cdots, e_n,e_{n+1}$, one has
\begin{eqnarray}\label{Eq.iso-emb}
\mu(\mathbb{R}^n)=n+1.
\end{eqnarray}
Define a new Borel measure $\nu$ on $S^n$, which is the push-forward of $\mu$ by $h$, by
\begin{align*}
\int_{S^n}g(u)d\nu(u)=\int_{\mathbb{R}^n}g(h(x))d\mu(x)
\end{align*}
for any Borel function $g: S^n\to\mathbb{R}$. Definition \ref{Def.iso-embedding} shows that $\nu$ is an isotropic measure on $S^n$, by applying the Ball-Barthe inequality to $\nu$, we can obtain
the following Ball-Barthe inequality for isotropic embeddings.
\begin{pro}\label{Pro.Ball-Barthe}
If $h:\mathbb{R}^n\rightarrow S^{n}$ is  an isotropic embedding of the Borel measure space $(\mathbb{R}^n,\mu)$ into $S^n$, then for each continuous $l: \mathbb{R}^n\rightarrow (0,\infty)$
\begin{eqnarray}\label{}
\det\int_{\mathbb{R}^n}l(x)h(x)\otimes h(x)d\mu(x)\geq \exp\left\{\int_{\mathbb{R}^n}\log l(x)d\mu(x)\right\},
\end{eqnarray}
with equality if and only if $l(x_1),\cdots,l(x_{n+1})$ is constants for  $x_1,\cdots, x_{n+1}$ in ${\rm supp}(\nu)$ such that $h(x_1),\cdots,h(x_{n+1})$ are linearly independent.
\end{pro}

\section{The LYZ ellipsoid of log-concave functions}\label{Sec.4}

In this section, we will define the LYZ ellipsoid of log-concave functions, and introduce its properties.

\begin{defi}\label{Def.LYZEforLC}
Let $f=e^{-\varphi}\in{\rm LC}_n$. Suppose that $\varphi^\ast(x)>0$ for any $x\in{\mathbb{R}^n\backslash\{o\}}$, then the \emph{LYZ ellipsoid} of $f$, denoted by $\Gamma_{-2}f$, is defined as
\begin{align}\nonumber
-\log\Gamma_{-2}f(x)&=\frac{n^2}{8\delta J(f,f)}\int_{\mathbb{R}^n}| x\cdot y|^2\varphi^*(y)^{-1}d\mu_f(y)\\
\label{Def.2}
&=\frac{n^2}{8\delta J(f,f)}\int_{\mathbb{R}^n}| x\cdot \nabla\varphi(y)|^2\varphi^*(\nabla\varphi(y))^{-1}e^{-\varphi(y)}dy.
\end{align}
\end{defi}

Note that the function $\sqrt{-\log\Gamma_{-2}f(x)}$ is a Minkowski functional of an ellipsoid. Indeed, if $E$ is an ellipsoid generated by $n\times n$ real symmetric matrix $A$, that is
\begin{align*}
E=\{x\in\mathbb{R}^n: x\cdot Ax\le 1\},
\end{align*}
one can check that the Minkowski function of $E$ is given by
$
\|x\|_E^2=x\cdot A x.
$
From this, we see that the function $\sqrt{-\log\Gamma_{-2}f(x)}$ is a Minkowski functional of an ellipsoid, which is generated by $n\times n$ real symmetric matrix $M=[m_{ij}(f)]$, where
\begin{align*}
m_{ij}(f)={\frac{n^2}{8\delta J(f,f)}}\int_{\mathbb{R}^n} (e_i\cdot y) (e_j\cdot y)\varphi^*(y)^{-1}d\mu_f(y).
\end{align*}

\begin{property}\label{affine}
Let $f\in{\rm LC}_n$ and $T\in {\rm GL}(n)$. Then
\begin{eqnarray}\label{Eq.GL}
\Gamma_{-2}(f\circ T)=(\Gamma_{-2}f)\circ T.
\end{eqnarray}
\end{property}
\begin{proof}
Let $T^t$ be the transpose of $T$. From the definition of Legendre transform (\ref{Fenchel conjugate}) and the fact that $ Tx\cdot y= x\cdot T^{t}y$, one has
\begin{eqnarray}\nonumber
(\varphi\circ T)^*(y)&=&\sup_{x\in\mathbb{R}^n}\left\{x\cdot y-\varphi(Tx)\right\}\\\nonumber
&=&\sup_{x\in\mathbb{R}^n}\left\{ x\cdot T^{-t}y-\varphi(x)\right\}\\\label{affine-legendre}
&=&\varphi^*(T^{-t}y).
\end{eqnarray}
By (\ref{Def.2}) and the fact that $\nabla (\varphi\circ T)(y)=T^t\nabla \varphi(Ty)$, one has
\begin{eqnarray*}\label{}
&&\int_{\mathbb{R}^n}| x\cdot y|^2(\varphi\circ T)^*(y)^{-1}d\mu_{f\circ T}(y)\\
&&\quad\quad =\int_{\mathbb{R}^n}| x\cdot \nabla(\varphi\circ T)(y)|^2(\varphi\circ T)^*(\nabla (\varphi\circ T)(y))^{-1}e^{-\varphi(Ty)}dy\\
&&\quad\quad=\int_{\mathbb{R}^n}| x\cdot T^t\nabla \varphi(Ty)|^2\varphi^*(\nabla\varphi(Ty))^{-1}e^{-\varphi(Ty)}dy\\
&&\quad\quad=\int_{\mathbb{R}^n}| Tx\cdot \nabla\varphi(Ty)|^2\varphi^*(\nabla\varphi(Ty))^{-1}e^{-\varphi(Ty)}dy\\
&&\quad\quad=|\det T|^{-1}\int_{\mathbb{R}^n}| Tx\cdot y|^2\varphi^*(y)^{-1}d\mu_{f}(y).
\end{eqnarray*}
This together with $ \delta J(f\circ T,f\circ T)=|\det T|^{-1}\delta J(f,f)$ gives (\ref{Eq.GL}).
\end{proof}
The following corollary shows that the quantity $|K_t(f^\circ)|J(\Gamma_{-2}(f))$ is invariant under the operation of $GL(n)$.

\begin{corollary}
Let $f\in{\rm LC}_n$ and $T\in {\rm GL}(n)$. Then for any $t>0$
$$
|K_t((f\circ T)^\circ)|J(\Gamma_{-2}(f\circ T))=|K_t(f^\circ)|J(\Gamma_{-2}(f)).
$$
\end{corollary}
\begin{proof}
By (\ref{affine-legendre}), one has for any $T\in GL(n)$,
\begin{align*}
K_t((f\circ T)^\circ)&=\{x\in\mathbb{R}^n: f^\circ(Tx)\ge t\}
=\{x\in\mathbb{R}^n:e^{-\varphi^*{(T^{-t}x)}}\ge t\}\\
&=\{T^ty\in\mathbb{R}^n:e^{-\varphi^*{(y})}\ge t\}=T^t(K_t(f^\circ)).
\end{align*}
Property (\ref{affine}) implies that
$$
J(\Gamma_{-2}(f\circ T))=J((\Gamma_{-2}f)\circ T)
=\int_{\mathbb{R}^n}(\Gamma_{-2}f)(Tx)dx=(\det T)^{-1}J(\Gamma_{-2}f).
$$
Hence
$$
|K_t((f\circ T)^\circ)|J(\Gamma_{-2}(f\circ T))=|(K_t(f^\circ))|J(\Gamma_{-2}f).
$$
This completes the proof.
\end{proof}

When $f=e^{-\frac{\|x\|_K^2}{2}}$ with $K\in \mathcal{K}_o^n$, $\Gamma_{-2}f(x)$ can be calculated precisely.
\begin{property}\label{Property.|x|K}
Let $K\in \mathcal{K}_o^n$. If $f=e^{-\frac{\|x\|_K^2}{2}}$, then
\begin{eqnarray*}\label{}
\Gamma_{-2}f(x)=e^{-\frac{\|x\|_{\Gamma_{-2}K}^2}{2}}.
\end{eqnarray*}
\end{property}

\begin{proof}
The normalized cone measure $\sigma_K$ of $K$ 
is defined  by
$$
d\sigma_{K}(z)=\frac{z\cdot n_{K}(z)}{n|K|}d\mathcal{H}^{n-1}(z) \quad \text{for}\quad z\in \partial K.
$$
For $y\in \mathbb{R}^n$,  we write $y=rz$ with $z\in \partial K$, then
\begin{align}\label{dy=dz}
dy=n|K|r^{n-1}drd\sigma_{K}(z).
\end{align}
Together with the fact that  $\left(\frac{\|x\|_{K}^2}{2}\right)^\ast=\frac{\|x\|_{K^\circ}^2}{2}$ for any $x\in\mathbb{R}^n$, one has
\begin{eqnarray*}
 &&\int_{\mathbb{R}^n}|x\cdot\nabla \varphi(y)|^2\varphi^*(\nabla \varphi(y))^{-1}e^{-\varphi(y)}dy\\
&&\quad\quad=\int_{\mathbb{R}^n} |x\cdot \|y\|_{K} \nabla  \|y\|_{K} |^2\left(\tfrac{\| \|y\|_{K} \nabla  \|y\|_{K} \|_{K^{\circ}}^2}{2}\right)^{-1}e^{-\tfrac{\|y\|_K^2}{2}}dy\\
&&\quad\quad=\int_{\mathbb{R}^n} |x\cdot  \nabla  \|y\|_{K} |^2\left(\tfrac{\|  \nabla  \|y\|_{K} \|_{K^{\circ}}^2}{2}\right)^{-1}e^{-\tfrac{\|y\|_K^2}{2}}dy\\
&&\quad\quad=2n|K|\int_0^{\infty}\int_{\partial K} r^{n-1}|x\cdot\nabla  \|z\|_{K}|^2\left(\|  \nabla  \|z\|_{K} \|_{K^{\circ}}^2\right)^{-1}e^{-\frac{r^2}{2}}drd\sigma_{K}(z)\\
&&\quad\quad=2^{\frac{n}{2}}\Gamma\left(\tfrac{n}{2}\right)n|K|\int_{\partial K}|x\cdot\nabla  \|z\|_{K}|^2\left(\|  \nabla  \|z\|_{K} \|_{K^{\circ}}^2\right)^{-1}d\sigma_{K}(z).
\end{eqnarray*}
By the fact that $\nabla \|z\|_{K}=\frac{n_{K}(z)}{\|n_{K}(z)\|_{K^{\circ}}}$ for $z\in\partial K$ (see \cite[Remark 1.7.14]{SchneiderR2014}), (\ref{Eq.Minkowski-support-radial}) and (\ref{radial-LYZ-ellipsoid}), we have
\begin{eqnarray*}
 &&n|K|\int_{\partial K}|x\cdot\nabla  \|z\|_{K}|^2\left(\|  \nabla  \|z\|_{K} \|_{K^{\circ}}^2\right)^{-1}d\sigma_{K}(z)\\
&&\quad\quad=n|K|\int_{\partial K}|x\cdot\nabla  \|z\|_{K}|^2d\sigma_{K}(z)\\
&&\quad\quad=n|K|\int_{\partial K}\left|x\cdot\tfrac{n_{K}(z)}{\|n_{K}(z)\|_{K^{\circ}}}\right|^2d\sigma_{K}(z)\\
&&\quad\quad=\int_{\partial K}\left|x\cdot n_{K}(z)\right|^2(z\cdot n_K(z))^{-1}d\mathcal{H}^{n-1}(z)\\
&&\quad\quad=|K|\rho_{\Gamma_{-2}K}(x)^{-2}\\
&&\quad\quad=|K|\|x\|_{\Gamma_{-2}K}^{2}.
\end{eqnarray*}
By using (\ref{definition of the first variation}) and (\ref{dy=dz}), we have
\begin{align*}
\delta J(f,f)\nonumber
&=\frac{1}{2}\int_{\mathbb{R}^n}{\| \|y\|_{K} \nabla  \|y\|_{K} \|_{K^{\circ}}^2}e^{-\tfrac{\|y\|_K^2}{2}}dy\\\nonumber
&=\frac{n|K|}{2}\int_0^\infty r^{n+1}e^{-\frac{r^2}{2}} dr\int_{\partial K}\|  \nabla  \|z\|_{K} \|_{K^{\circ}}^2d\sigma_K(z)\\
&=2^{\frac{n}{2}-2}\Gamma(\tfrac{n}{2})n^2|K|.
\end{align*}
Hence
\begin{eqnarray*}
 -\log \Gamma_{-2}f(x)=\frac{\|x\|_{\Gamma_{-2}K}^{2}}{2}.
\end{eqnarray*}
This completes the proof.
\end{proof}

When $f$ is the Gaussian function $e^{-\frac{|x|^2}{2}}$, the following result can be obtained directly  from Property \ref{Property.|x|K} and the fact that $\Gamma_{-2} B^n=B^n$.
\begin{cor}
If $f=\gamma_n=e^{-\frac{|x|^2}{2}}$, then
\begin{eqnarray*}\label{}
\Gamma_{-2}\gamma_n=\gamma_n.
\end{eqnarray*}
\end{cor}

\section{LYZ polar inequality for log-concave functions}\label{Sec.5}
In this section, we will establish the LYZ polar inequality for log-concave functions. The following lemmas are needed.
\begin{lemma}\label{Le.contain}
Let $f=e^{-\varphi}\in{\rm LC}_n$ and $K_t(f)=\{x\in\mathbb{R}^n:f(x)\ge t\}$ be the upper level set of $f$. If $\nu$ is a finite, positive Borel  measure on $\mathbb{R}^n$ and $\varphi^\ast>0$ on $\mathbb{R}^n\backslash\{o\}$, then
\begin{eqnarray*}\label{}
\int_{\mathbb{R}^n}x\varphi^\ast(x)^{-1}d\nu(x)\in r K_t(f^\circ),
\end{eqnarray*}
where $-\log t=\sup\{\varphi^\ast(x/\varphi^\ast(x)): {x\in{\rm supp}(\nu)}\}$ and $r=|\nu|$.
\end{lemma}
\begin{proof}

Let
\begin{eqnarray*}\label{}
x_0=\int_{\mathbb{R}^n}x\varphi^\ast(x)^{-1}d\nu(x).
\end{eqnarray*}
By the convexity of $\varphi^\ast$, we have
\begin{eqnarray*}\label{}
\varphi^\ast\left(\frac{x_0}{r}\right)\leq\int_{\mathbb{R}^n}\varphi^\ast\left(\frac{x}{\varphi^\ast(x)}\right)\frac{d\nu(x)}{r}
\le -\log t.
\end{eqnarray*}
Hence
\begin{eqnarray*}\label{}
x_0\in r\left\{x\in \mathbb{R}^n: f^\circ(x)\geq t\right\}.
\end{eqnarray*}
Then $x_0\in r K_t(f)$.
\end{proof}
Note that $-\log t$ may reach infinity, but it is often bounded for many log-concave functions, such as $f=e^{-\varphi}\in {\rm LC}_n$ with
$\varphi^{\ast} (x)>0$ for $x\in  \mathbb{R}^n\backslash\{o\}$ and ${\rm dom}(\varphi^\ast)=\mathbb{R}^n$.
We first need to show  $J(f^\circ)<\infty$.
For the Legendre transformation of $\varphi$,
$$
\varphi^\ast(y)=\sup_{x\in\mathbb{R}^n} x\cdot y-\varphi(x)\ge |y|-\varphi\left(\frac{y}{|y|}\right),
$$
 we have $\varphi^\ast(y)\to\infty$ as $|y|\to\infty$ and  hence $J(f^\circ)<\infty$. Moreover, $f^\circ\in{\rm LC}_n$. \cite[Lemma 2.5]{CF13} shows that there exists two constants $a>0$ and $b$ such that
\begin{align}
\varphi^\ast(x)\ge a|x|+b
\end{align}
for any $x\in\mathbb{R}^n$. Hence
$$
\lim_{|x|\to\infty}\frac{|x|}{\varphi^\ast(x)}\le\lim_{|x|\to\infty}\frac{|x|}{a|x|+b}
=\frac{1}{a},
$$
that is, for any $\varepsilon>0$, there exists a constant $M$ such that ${x}/{\varphi^\ast(x)}\in \left(\frac{1}{a}+\varepsilon\right)B^n$
for  $x\in (M B^n)^c$ and  the unit ball $B^n$ in $\mathbb{R}^n$. Then $\varphi^\ast\left({x}/{\varphi^\ast(x)}\right)$ is bounded for any $x\in (M_0 B^n)^c$ as the continuity of $\varphi^\ast$. Set $m=\min\{\varphi^\ast(x):x\in M B^n\}$, then
$x/\varphi^\ast(x)\in (M/m)B^n$ for any $x\in M B^n$. Hence  $\varphi^\ast\left({x}/{\varphi^\ast(x)}\right)$ is bounded on $MB^n$ and $-\log t$ is a finite constant.

The following lemma is the key for the equality condition of our main inequality.
\begin{lemma}\label{Le.eqcase}
Let $\varphi$ be convex function and {\rm dom}$(\varphi)=\mathbb{R}^n$ and $\varphi(x)>\varphi(o)$ for all $x\in\mathbb{R}^n\backslash\{o\}$. Suppose that $\varphi$ satisfies the equality
\begin{align}\label{Eq.MinkowskiF}
\varphi(x)=x\cdot \nabla \varphi(x)-c
\end{align}
a.e. on $\mathbb{R}^n$, where $c$ is a constant. Then $\varphi(x)=\|x\|_K-c$ for some $K\in\mathcal{K}_o^n$.
\end{lemma}
\begin{proof}
Firstly, we consider the case when $c=0$, in which case we claim that
\begin{align}\label{linear}
\varphi(\lambda x)=\lambda\varphi(x)
\end{align}
for all $x\in\mathbb{R}^n$ and any $\lambda\ge 0$. Since $|\nabla\varphi(x)|$ is finite near the origin, hence $|\varphi(o)|=\lim_{|x|\to 0}|x\cdot\nabla\varphi(x)|=0$, which implies that (\ref{linear}) holds for $\lambda=0$. Otherwise, if $\varphi(\lambda x)\neq\lambda\varphi(x)$ for some points $x\neq o$, where $\varphi$ is differentiable, then by the convexity of  $\varphi(x)$, one has
\begin{align*}
\varphi(y)-\varphi(x)\ge\nabla\varphi(x)\cdot(y-x)
\end{align*}
for all $y \in\mathbb{R}^n$.
Note that the above inequality is strictly when $y=o$ due to the assumptions $\varphi(x)>\varphi(o)$ and $\varphi(\lambda x)\neq\lambda\varphi(x)$.
Then by taking $y=o$, we have $\varphi(x)<\nabla\varphi(x)\cdot x$, which  contradicts with (\ref{Eq.MinkowskiF}). Thus
$\varphi(\lambda x)=\lambda\varphi(x)$ at every differentiable points. Moreover, this also holds on whole $\mathbb{R}^n$ as $\varphi$ is continuous and $\varphi$ is differentiable almost everywhere.
Combined with the convexity of $\varphi$, one sees $\varphi$ is sublinear.
By the fact that every sublinear function from $\mathbb{R}^n$ to $\mathbb{R}$ is a support function of a convex body, we have $\varphi(x)=h_K(x)$ for some convex body $K$.
Since $\varphi(x)>\varphi(o)=0$ for all $o\neq x\in\mathbb{R}^n$, one can see that $K$ is a convex body that contains the origin in its interior. In this case, one has $\varphi(x)=h_{K}(x)=\|x\|_{K^\circ}$ and (\ref{Eq.MinkowskiF}) follows.
\end{proof}

For a real function $\tau : (0,\ +\infty) \rightarrow R$, let
\begin{eqnarray}\label{def.tau}
\int_0^{\tau(t)}e^{-l}dl=\frac{1}{\sqrt{\pi}}\int_{-\infty}^te^{-l^2}dl.
\end{eqnarray}
The  transformation $T: \mathbb{R}^{n+1}\rightarrow\mathbb{R}^{n+1}$ is defined by
\begin{eqnarray*}\label{}
Ty=\int_{\mathbb{R}^n}h(x)\frac{\tau(y\cdot h(x))}{h_i(x)}d\nu(x),
\end{eqnarray*}
for  $y\in\mathbb{R}^{n+1}$, with $h_i=h\cdot e_i$  for $i=1,\cdots,n+1$.

By the Ball-Barth inequality, we obtain the following lemma.

\begin{lemma}\label{Le.key}
Let $(\mathbb{R}^n,\nu)$ be a Borel measure space and $h: \mathbb{R}^n\to S^n$ be an isotropic embedding of the Borel measure space $(\mathbb{R}^n,\nu)$ into $S^n$.  Then
\begin{align}\label{Key.In}
(n+1)^{\frac{n+1}{2}}\le\int_{T(\mathbb{R}^{n+1})}e^{-x_i}dx,
\end{align}
with equality if and only if $h_i$ is constant on ${\rm supp} (\nu)$, and there exists a $C>0$ with respect to $y$ such that
$$
\prod_{j=1}^{n+1}{\tau^\prime(y\cdot h(x_i))}=C
$$
for $x_1,\cdots, x_{n+1} \in {\rm supp} (\nu)$ with $h(x_1), . . . , h(x_{n+1})$  linearly independent.

\end{lemma}

\begin{proof}
Deriving both sides of (\ref{def.tau}) with respect to $t$, we have
\begin{eqnarray*}\label{}
-\tau(t)+\log \tau'(t)=-\log\sqrt{\pi}-t^2.
\end{eqnarray*}
Taking $t=y\cdot h(x)$, one has
\begin{eqnarray}\label{eq.1}
-|y\cdot h(x)|^2=\log\sqrt{\pi}-\tau(y\cdot h(x))+\log\frac{\tau'(y\cdot h(x))}{h_i(x)}+\log h_i(x).
\end{eqnarray}
Now we integrate (\ref{eq.1}) over all $x\in \mathbb{R}^n$ with respect to the measure $d\nu$.
From the definition of isotropic embedding (\ref{def.eq.iso-em}), the integral of the left hand side on (\ref{eq.1}) is equal to
\begin{eqnarray}\label{In.1}
-\int_{\mathbb{R}^n}|y\cdot h(x)|^2|d\nu(x)=-|y|^2.
\end{eqnarray}
Next we deal with the integral of the last term on the right hand side of (\ref{eq.1}). From (\ref{Eq.iso-emb}), one can see that the measure $\frac{1}{n+1}d\nu$ is a probability measure. By the fact that the $L_0$-mean of a function is dominated by its $L_2$-mean on a probability space and (\ref{def.eq.iso-em}), one has
\begin{eqnarray}\label{In.1.1}
\exp\left(\frac{1}{n+1}\int_{\mathbb{R}^n}\log h_i(x)d\nu(x)\right)\leq \left(\frac{1}{n+1}\int_{\mathbb{R}^n}|h_{i}(x)|^2d\nu(x)\right)^{\frac{1}{2}}
=\left(\frac{1}{n+1}\right)^{\frac{1}{2}},
\end{eqnarray}
with equality if and only if $h_i(x)$ is constant on ${\rm supp}(\nu)$. Hence
\begin{eqnarray*}\label{}
\int_{\mathbb{R}^n}\log h_i(x)d\nu(x)\leq\log\left(\frac{1}{n+1}\right)^{\frac{n+1}{2}}.
\end{eqnarray*}

Calculating the derivative of $Ty$ with respect to $y$, we have
\begin{eqnarray*}\label{}
dTy=\int_{\mathbb{R}^n}h(x)\otimes h(x)\frac{\tau'(y\cdot h(x))}{h_i(x)}d\nu(x).
\end{eqnarray*}
From Proposition \ref{Pro.Ball-Barthe}, we infer that
\begin{eqnarray}\label{In.1.2}
\det(dTy)&\geq&\exp\left\{\int_{\mathbb{R}^n}\log \frac{\tau'(y\cdot h(x))}{h_i(x)}d\nu(x)\right\},
\end{eqnarray}
with equality if and only if
\begin{eqnarray*}\label{}
\prod_{j=1}^{n+1} \frac{\tau'(y\cdot h(x_j))}{h_i(x_j)}
\end{eqnarray*}
is constant for $x_1,\cdots,x_{n+1}\in {\rm supp}(\mu)$ such that $h(x_1),\cdots,h(x_{n+1})$ are linearly independent.

Combining  (\ref{eq.1}), (\ref{In.1}), (\ref{In.1.1}), (\ref{In.1.2}) and the fact that $\nu(\mathbb{R}^n)=n+1$, we have
\begin{align*}
\exp\{-|y|^2\}\le \left(\frac{\pi}{n+1}\right)^{\frac{n+1}{2}}\det(d(Ty))
\exp\{-e_{i}\cdot Ty\}.
\end{align*}
Integrating this inequality over all $y\in \mathbb{R}^{n+1}$ gives the desired inequality (\ref{Key.In}).%
\end{proof}

Now we are ready to prove our functional isoperimetric  inequality for log-concave functions.

\begin{thm}\label{mainthm}
Let $f=e^{-\varphi}\in{\rm LC}_n$ and $\mu_f$ be the surface area measure of $f$. Suppose that $\varphi^\ast(x)>0$ on $\mathbb{R}^n\backslash\{o\}$ and $-\log t=\sup\{\varphi^\ast(x/\varphi^\ast(x)):x\in {\rm supp} (\mu_f)\}$, then
\begin{eqnarray}\label{mainIn}
|K_{t}(f^\circ)|J(\Gamma_{-2}f)
\geq\frac{8^{\frac{n}{2}}(n+1)^{\tfrac{n+1}{2}}\Gamma(\frac{n}{2}+1)\omega_n}{n!n^n}.
\end{eqnarray}
Moreover, if $\varphi(x)>\varphi(o)$ for all $x\in\mathbb{R}^n\backslash\{o\}$, then equality holds if and only if $f(x)=e^{-{\|x\|_K}+1}$ and $K$ is a simplex whose centroid is at the origin.
\end{thm}
\vskip 0.2cm
\begin{proof}
From Property \ref{affine}, (\ref{Eq.(TK)o}) and the fact that $\sqrt{-\log \Gamma_{-2}f(x)}$ is a Minkowski function of an ellipsoid, we only need to prove (\ref{mainIn}) holds for $\Gamma_{-2}f(x)=e^{-{|x|^2}/{2}}$. This means that the measure
$$
\nu(\cdot)=\frac{n^2}{4\delta J(f,f)}(\varphi^*)^{-1}\mu_f(\cdot)
$$
is isotropic.
Let $c_n={2/n}$. For any $x\in\mathbb{R}^n$, define a function $h: \mathbb{R}^{n}\rightarrow\mathbb{R}^{n+1}$ by
$
h(x)=(x,c_n\varphi^\ast(x)),
$
and $\bar{h}: \mathbb{R}^{n}\rightarrow S^n$ by
$
\bar{h}={h}/{|h|}.
$
Assume that $y=(z,s)\in\mathbb{R}^{n+1}$. By the fact that the barycenter of $\mu_f$ is at the origin, we have
\begin{eqnarray*}\label{}
&&\int_{\mathbb{R}^n}|y\cdot \bar{h}(x)|^2|h(x)|^2d\nu(x)\\
&&\quad\quad=\int_{\mathbb{R}^n}|(z,s)\cdot (x,c_n \varphi^\ast(x))|^2d\nu(x)\\
&&\quad\quad=\int_{\mathbb{R}^n}|z\cdot x|^2d\nu(x)+2sc_n z\cdot\int_{\mathbb{R}^n} x d\mu_f(x)+\frac{4s^2}{n^2}\int_{\mathbb{R}^n}(\varphi^\ast)^2d\nu(x)\\
&&\quad\quad=|z|^2+s^2\\
&&\quad\quad=|y|^2.
\end{eqnarray*}
Hence $\bar{h}: \mathbb{R}^n\rightarrow S^n$ is an isotropic embedding of the Borel measure space $(\mathbb{R}^n, |h|^2d\nu)$ into $S^n$.
From Lemma \ref{Le.key}, one can see that
\begin{eqnarray}\label{}
(n+1)^{\frac{n+1}{2}}\leq\int_{\mathbb{R}^{n+1}}e^{-e_{n+1}\cdot Ty}|dTy|dy=\int_{T(\mathbb{R}^{n+1})}e^{-e_{n+1}\cdot z}dz.
\end{eqnarray}
Here $Ty$ is given by
\begin{align*}\label{}
Ty&=\int_{\mathbb{R}^n}\bar{h}(x)\frac{\tau(y\cdot \bar{h}(x))}{e_{n+1}\cdot \bar{h}(x)}|h(x)|^2d\nu(x)\\
&=\int_{\mathbb{R}^n}h(x)\frac{\tau(y\cdot \bar{h}(x))}{e_{n+1}\cdot h(x)}|h(x)|^2d\nu(x)\\
&=\int_{\mathbb{R}^n}((c_n\varphi^\ast(x))^{-1} x,1)\tau(y\cdot \bar{h}(x))|h(x)|^2d\nu(x).
\end{align*}
Lemma \ref{Le.contain} implies that
\begin{eqnarray*}\label{}
Ty\in \bigcup_{r>0} c_n^{-1}rK_{{t}}(f^\circ)\times\{r\}=:D.
\end{eqnarray*}
Therefore,
\begin{align*}\label{}
\int_{T(\mathbb{R}^{n+1})}e^{-e_{n+1}\cdot z}dz
&\leq\int_{D}e^{-e_{n+1}\cdot z}dz
=\int_0^{\infty}\int_{c_n^{-1}rK_{t}}e^{-r}dxdr\\
&=\frac{n^n}{2^n}\int_0^{\infty}r^ne^{-r}dr|K_{t}(f^\circ)|
=\frac{n^n}{2^n}\Gamma(n+1)| K_{t}(f^\circ)|.
\end{align*}
This together with the fact that $J(e^{-\frac{|x|^2}{2}})=2^{\frac{n}{2}}\Gamma(\frac{n}{2}+1)\omega_n$ gives (\ref{mainIn}).

From the equality condition of inequality (\ref{Key.In}), we have that $\varphi^\ast$ is a positive constant on ${\rm supp} (\tau(y\cdot\overline{h})|h|^2\nu)$, that is, $\varphi^\ast=c_0>0$ on ${\rm supp} (\mu_f)$ due to $\tau>0$ on $\mathbb{R}$ and $|h(x)|^2>0$ on $\mathbb{R}^{n+1}$. Since ${\rm dom}(\varphi)=\mathbb{R}^n$, one can see that
\begin{align}
\varphi^\ast(\nabla\varphi(x))
=x\cdot\nabla\varphi(x)-\varphi(x)=c_0
\end{align}
almost everywhere on $\mathbb{R}^n$. If $\varphi(x)>\varphi(o)$ for any $x\in\mathbb{R}^n\backslash\{o\}$, Lemma \ref{Le.eqcase} implies that $\varphi(x)=\|x\|_K-c_0$ for some convex body $K\in\mathcal{K}^n_o$.
In this case, by (\ref{dy=dz}), one can calculate that
\begin{align}\label{Eq.J(f,f)}
\delta J(f,f)=c_0 e^{c_0}n\Gamma(n)|K|
\end{align}
and
\begin{align*}
&\int_{\mathbb{R}^n}|x\cdot\nabla \varphi(y)|^2\varphi^*(\nabla \varphi(y))^{-1}e^{-\varphi(y)}dy\\
&\quad\quad= c_0^{-1}e^{c_0}\int_{\mathbb{R}^n}|x\cdot\nabla \|y\|_K|^2e^{-\|y\|_K}dy\\
&\quad\quad= c_0^{-1}e^{c_0}n|K|\int_{0}^\infty\int_{\partial K}|x\cdot\nabla \|z\|_K|^2r^{n-1}e^{-r}drd\sigma_K(z)\\
&\quad\quad=c_0^{-1}e^{c_0}\Gamma(n)n|K|\int_{\partial K}|x\cdot\nabla \|z\|_K|^2drd\sigma_K(z)\\
&\quad\quad=c_0^{-1}e^{c_0}\Gamma(n)|K|\|x\|_{\Gamma_{-2}K}^{2}.
\end{align*}
This together with the definition of $-\log\Gamma_{-2}(f)$ and (\ref{Eq.J(f,f)}) shows that
\begin{align*}
-\log \Gamma_{-2}(e^{-\|x\|_K+c_0})=
\frac{n}{8c_0^2}\|x\|_{\Gamma_{-2}K}^{2}.
\end{align*}
For any $c>0$ and convex body $K$,
\begin{align*}\label{}
J(e^{-c\|x\|_K^2})&=\int_{\mathbb{R}^n}e^{-c\|x\|_K^2}dx\\
&=n|K|\int_0^{\infty}\int_{\partial K}r^{n-1}e^{-cr^2}drd\sigma_K(x)\\
&=c^{-\tfrac{n}{2}}\Gamma(\tfrac{n}{2}+1)|K|.
\end{align*}
Hence
\begin{align*}
J(\Gamma_{-2}(e^{-\|x\|_K+c}))=\frac{8^\frac{n}{2}c_0^n}{n^\frac{n}{2}}\Gamma(\tfrac{n}{2}+1)|\Gamma_{-2} K|.
\end{align*}
By the definition of Legendre transform and $\varphi(x)=\|x\|_K-c_0$, we have
\begin{equation*}
\label{eq6}
\varphi^\ast(x)=\left\{
\begin{aligned}
&\ c_0,& x\in K^\circ; \\
&+\infty,  &x\notin K^\circ .
\end{aligned}
\right.
\end{equation*}
For any $x\in {\rm supp}(\mu_f)\subset K^\circ$ and $c_0\ge1$, we obtain
$$-\log t=\sup\left\{\varphi^\ast\left(\frac{x}{\varphi^\ast(x)}\right):x\in{\rm supp}{(\mu_f)}\right\}
=c_0.$$
Hence
\begin{align*}
K_{t}(f^\circ)
=\{x\in\mathbb{R}^n: \varphi^\ast(x)\le c_0\}=K^\circ
\end{align*}
and the inequality (\ref{mainIn}) becomes
\begin{align}\label{c_0}
|K^\circ||\Gamma_{-2}K|
\geq\frac{(n+1)^{\tfrac{n+1}{2}}\omega_n}{c_0^nn!n^{\frac{n}{2}}}.
\end{align}

Let $conv(A)$ denote the convex hull of $A \subset  \mathbb{R}^n$ and  $   \overline{A}$ denote the closure of $A$.
For the surface area  $\mu_f$ of the log-concave function, \cite[(16)]{CK15} says that
$$
{\rm conv}({\rm supp}(\mu_f))=\overline{\{\varphi^\ast<+\infty\}},
$$
therefore  ${\rm conv}({\rm supp}(\mu_f))=K^\circ$.

For  $x\in \partial({\rm supp}(\mu_f))\subset \partial(K^\circ)$ and $c_0<1$, we have  $x/c_0\notin K^\circ$ and  $-\log t=+\infty$.
 Therefore $|K_{t}(f^\circ)|=+\infty$ and the inequality (\ref{mainIn}) is strict. Hence the best constant of (\ref{c_0}) is $c_0=1$ and the inequality (\ref{c_0}) becomes  the LYZ  polar inequality (\ref{In.LYZpolar}). Therefore the equality of (\ref{c_0}) holds if and only if $c_0=1$ and
 $K$ is a simplex with centroid at the origin.
\end{proof}

\vskip 0.5 cm


\end{document}